\documentclass[reqno,12pt]{amsart}     
 
\usepackage{amssymb, amsmath, amsthm, epsf, epsfig}

\renewcommand{\(}{\left(}
\renewcommand{\)}{\right)}

\newtheorem{theo}{Theorem}
\newtheorem{prop}[theo]{Proposition}
\newtheorem{lemma}[theo]{Lemma}
\newtheorem{cor}[theo]{Corollary}

\theoremstyle{definition}

\newtheorem*{df}{Definition}
\newtheorem*{ex}{Example}
\newtheorem*{exs}{Examples}

\theoremstyle{remark}

\newtheorem*{rem}{Remark}

\newcommand{\beq}{\begin{equation}} 
\newcommand{\eeq}{\end{equation}} 
\newcommand{\bal}{\begin{align}} 
\newcommand{\eal}{\end{align}} 
\newcommand{\bals}{\begin{align*}} 
\newcommand{\eals}{\end{align*}} 
\newcommand{\barr}[1]{\begin{array}{#1}} 
\newcommand{\earr}{\end{array}}

\newcommand{\bth}{\begin{theo}} 
\newcommand{\bl}{\begin{lemma}} 
\newcommand{\el}{\end{lemma}} 
\newcommand{\bp}{\begin{prop}} 
\newcommand{\ep}{\end{prop}} 
\newcommand{\bdf}{\begin{df}} 
\newcommand{\edf}{\end{df}} 
\newcommand{\brem}{\begin{rem}} 
\newcommand{\erem}{\end{rem}} 
\newcommand{\bnrem}{\begin{nrem}} 
\newcommand{\enrem}{\end{nrem}} 
\newcommand{\bex}{\begin{ex}} 
\newcommand{\eex}{\end{ex}} 
\newcommand{\bcor}{\begin{cor}} 
\newcommand{\ecor}{\end{cor}} 
\newcommand{\bncor}{\begin{ncor}} 
\newcommand{\encor}{\end{ncor}} 
\newcommand{\bpf}{\begin{proof}} 
\newcommand{\epf}{\end{proof}}

\def\({\left(} 
\def\){\right)} 
\def\fl#1{\left\lfloor#1\right\rfloor}

\def\Cov{\operatorname{\bf Cov}}
\newcommand{\E}{\operatorname{\bf E}} 
\newcommand{\PP}{\operatorname{\bf P}} 
\newcommand{\V}{\operatorname{\bf Var}} 
\def\ep{\varepsilon}
 
\numberwithin{equation}{section}

\begin{document}

\newbox\Adr
\setbox\Adr\vbox{
\centerline{$^*$Institute of Discrete Mathematics and Geometry, TU Wien,} 
\centerline{Wiedner Hauptstra\ss e 8--10, A-1040 Vienna, Austria.}
\centerline{WWW: \footnotesize\tt https://www.dmg.tuwien.ac.at/drmota}
\vskip6pt
\centerline{$^\dagger$Fakult\"at f\"ur Mathematik der Universit\"at Wien,}
\centerline{Oskar-Morgenstern-Platz~1, A-1090 Vienna, Austria.}
\centerline{WWW: \footnotesize\tt
  http://www.mat.univie.ac.at/\lower0.5ex\hbox{\~{}}kratt} 
}

\title[A Joint Central Limit Theorem for the Sum-of-Digits Function]
{A Joint Central Limit Theorem for the Sum-of-Digits Function,
and Asymptotic Divisibility of Catalan-like Sequences}

\author[Michael Drmota and Christian Krattenthaler]{Michael Drmota$^*$
  and Christian Krattenthaler\textsuperscript{\dag{}}\\\\\box\Adr} 

\address{$^*$Institute of Discrete Mathematics and Geometry, TU Wien, 
Wiedner Hauptstra\ss e 8--10, A-1040 Vienna,
Austria. WWW: {\tt https://www.dmg.tuwien.ac.at/drmota}.}

\address{\textsuperscript{\dag{}}Fakult\"at f\"ur Mathematik,
  Universit\"at Wien,  
Oskar-Morgenstern-Platz~1, A-1090 Vienna, Austria.
WWW: {\tt http://www.mat.univie.ac.at/\lower0.5ex\hbox{\~{}}kratt}.}

\thanks{{}$^{*\dagger}$\,Research partially 
supported by the Austrian  Science Foundation FWF, 
in the framework of the Special Research Program
``Algorithmic and Enumerative Combinatorics", project F50-N15.}



\begin{abstract}
We prove a central limit theorem for the 
joint distribution of $s_q(A_jn)$, $1\le j \le d$, where $s_q$ denotes
the sum-of-digits function in base~$q$ and the $A_j$'s are positive integers
relatively prime to~$q$. We do this in fact within the framework
of quasi-additive functions.
As application, we show that most elements of 
``Catalan-like" sequences --- by which we mean integer sequences
defined by products/quotients of factorials --- are divisible by
any given positive integer.
\end{abstract}

\maketitle 
 
\section{Introduction}

In \cite{BurnAA}, Burns shows that most of the ubiquitous 
{\it Catalan numbers} $C_n:=\frac {1} {n+1}\binom {2n}n$
(cf.\ \cite[Ex.~6.19]{StanBI})
are divisible by~$p$, where $p$ is some given prime number.
Let $v_p(N)$ denote the {\it $p$-adic valuation} of the integer~$N$,
which by definition is the maximal exponent~$\alpha$ such that
$p^\alpha$ divides~$N$.
In view of Legendre's formula \cite[p.~10]{LegeAA}
for the $p$-adic valuation of factorials, 
\begin{equation} \label{eq:Leg} 
v_p(n!)=\frac {1} {p-1}\big(n-s_p(n)\big),
\end{equation}
where $s_p(N)$ denotes the {\it $p$-ary sum-of-digits function}
\[
s_p(N) = \sum_{j\ge 0} \varepsilon_j(N),
\]
with $\varepsilon_j(N)$ denoting the $j$-th digit in the $p$-adic
representation of~$N$, we have
$$
v_p\left(\frac {1} {n+1}\binom {2n}n\right)
=\frac {1} {p-1}\big(2s_p(n)-s_p(2n)\big)-v_p(n+1).
$$
Thus, one sees that the above and many more asymptotic divisibility
results --- such as the divisibility of most of the Catalan numbers,
or even of most of the
{\it Fu\ss--Catalan numbers} (cf.\ \cite[pp.~59--60]{ArmDAA})
by any given prime {\it power} --- 
can be proved if one has sufficiently precise results
on the distribution of the vector
\begin{equation} \label{eq:vector} 
(s_p(A_1n),s_p(A_2n),\dots,s_p(A_dn)),\quad \quad n<N.
\end{equation}
Indeed, for $p=2$,
Schmidt \cite{Schmidt} and Schmid \cite{Schmid} 
showed that for pairwise different positive odd integers
$A_1,A_2,\ldots, A_d$ the vector \eqref{eq:vector} satisfies
a $d$-dimensional central limit theorem 
with asymptotic mean vector $(1/2,\ldots, 1/2) \cdot \log_2 N$ and
asymptotic covariance matrix $\Sigma \cdot\log_2 N$ with 
\[
\Sigma = \left(  \frac{{\rm gcd}(A_i,A_j)^2}{4 A_i A_j}
\right)_{1\le i,j \le d}. 
\]

We want to mention that the distribution of the vector \eqref{eq:vector}
in residue classes was intensively studied in \cite{DaTe}. 
Furthermore we also emphasize that many results on divisiblity
properties of non-central 
binomial coefficients are available, too, see for example \cite{SpW}.

\medskip

The first goal of the present paper is to generalize the central limit
theorem by Schmidt \cite{Schmidt} and Schmid \cite{Schmid}
to arbitrary primes~$p$, and even to arbitrary bases~$q$.
We do this in Theorem~\ref{ThCLT} 
in Section~\ref{sec1}, by using an even more general 
concept, namely the concept of {\it $q$-quasi-additive functions}.

We then apply this result in
Section~\ref{sec2} (see Theorem~\ref{thm:almostall} and
Corollary~\ref{cor:almostall}) 
to prove the somewhat non-intuitive fact that most elements of 
any sequence $\big(S(n)\big)_{n\ge0}$ of integers given by a
(non-trivial) formula
\begin{equation*} 
\frac {P(n)} {Q(n)}
\frac {
\prod _{i=1} ^{r}(C_in)!} {
\prod _{i=1} ^{s}(D_in)!} 
\end{equation*}
are divisible by any given prime power, 
and thus by any given positive integer.
Here, $P(n)$ and $Q(n)$ are polynomials in~$n$ over
the integers, where $Q(n)$ is a product of linear factors,
and the $C_i$'s and $D_i$'s are positive integers with
$\sum_{i=1}^rC_i=\sum_{i=1}^sD_i$. The attribute ``non-trivial"
means that the set of~$C_i$'s is different from the set of~$D_i$'s.
As is pointed out in
more detail in Section~\ref{sec2}, numerous (mainly
combinatorial) sequences that appear in the literature in
various contexts are of this form.

\section{A central limit theorem}
\label{sec1}

Let $q\ge 2$ be a given integer. 
It is well known that the sum-of-digits function 
$s_q(n)$ satisfies a central
limit theorem of the form
\begin{equation} \label{eq:clt} 
\frac 1N \# \left\{ n < N : s_q(n) \le \mu_q \log_q N + t \sqrt{\sigma_q^2
  \log_q N} \right\}  
= \Phi(t) + o(1),
\end{equation}
uniformly in $t$, where $\mu_q = (q-1)/2$, $\sigma_q^2 = (q^2-1)/12$, and 
$\Phi(t)$ denotes the distribution function of the standard Gau\ss ian
distribution. 
This result is easy to prove since the digits $\varepsilon_j(n)$,
$0\le j < \log_q(N)$, 
behave almost as i.i.d.\ random variables if $n$ varies between $0$ and $N-1$.
Actually much more is known (see for example \cite{BK}). Suppose that
$P(x)$ is a 
polynomial of degree $D\ge 1$ with non-negative integer
coefficients. Then we also have 
\[
\frac 1N \# \left\{ n < N : s_q(P(n)) \le \mu_q \log_q P(N) + t
\sqrt{\sigma_q^2 \log_q P(N)} \right\}  
= \Phi(t) + o(1).
\]
Note that the value $P(N)$ can be replaced by $N^D$ without changing
the validity of the statement. 

This result applies in particular to linear polynomials $P_j(n) = A_j
n$ (with integers $A_j \ge 1$). 
In what follows, we will consider linear combinations of the form
\begin{equation}\label{eqdeffn}
f(n) = c_1 s_q(A_1n) + c_2 s_q(A_2n) + \cdots + c_d s_q(A_dn), \qquad n < N,
\end{equation}
with real numbers $c_j$ and integers $A_j \ge 1$, $1\le j\le d$.
Clearly, the central limit result of
Schmidt \cite{Schmidt} and Schmid \cite{Schmid} mentioned in the
introduction is equivalent to the fact that 
$f(n)$ as in \eqref{eqdeffn} with $q=2$, $n < N$, satisfies a 
one-dimensional central limit theorem with asymptotic mean $\frac
12(c_1 + c_2 + \cdots + c_d)\cdot \log_2 N$ and 
asymptotic covariance $ {\bf c} \Sigma {\bf c}^t \cdot \log_2 N$,
where ${\bf c} = (c_1,c_2, \ldots, c_d)$. 

It is also clear that the results of \cite{Schmid,Schmidt} should directly
transfer to a general basis $q\ge 2$ 
so that we can cover general $f(n)$. We will establish this
generalization, however, with a  
completely different (and in fact more modern) proof. 

\begin{theo}\label{ThCLT}
Let $q\ge 2$ be an integer, 
and let $A_1,A_2,\ldots, A_d$ be positive integers. Then the vector 
\begin{equation}\label{eqTh}
(s_q(A_1 n), s_q(A_2 n), \ldots, s_q(A_d n) ), \qquad 0\le n < N,
\end{equation}
satisfies a $d$-dimensional central limit theorem
with asymptotic mean vector $((q-1)/2,\break\ldots, (q-1)/2) \cdot \log_q N$
and asymptotic covariance matrix $\Sigma\cdot \log_q N$, where
$\Sigma$ is positive semi-definite.

If we further assume that $q$ is prime and that the integers
$A_1,A_2,\ldots, A_d$ 
are not divisible by $q$, then $\Sigma$ is explicitly given by
\begin{equation} \label{eq:Cov} 
\Sigma = \left( \frac {(q^2-1)}{12} \frac{{\rm gcd}(A_i,A_j)^2}{A_i A_j}
\right)_{1\le i,j \le d}. 
\end{equation}
\end{theo}

For the proof we make use of the (recent) concept of quasi-additivity
which is thoroughly discussed 
in \cite{quasi}. There, a function $f$ defined on the non-negative
integers is called 
{\it $q$-quasi-additive}, if there exists $r\ge 0$ such that
\begin{equation}\label{eqquasi}
f(q^{k+r}a+b)=f(a)+f(b) \qquad \mbox{for all $b < q^k$}.
\end{equation}
We note that if \eqref{eqquasi} holds for some $r\ge 0$, then it holds
as well for every larger $r$.
This also shows that linear combinations of $q$-quasi-additive
functions are $q$-quasi-additive, too.
We further note that $s_q(n)$ is $q$-quasi-additive with parameter $r=0$.

One of the main results of the paper \cite{quasi} is that any 
$q$-quasi-additive function $f(n)$ of at most logarithmic growth
satisfies a central limit theorem of the form
\[
\frac 1N \# \left\{ n < N : f(n) \le \mu \log_q N + t \sqrt{\sigma^2 \log_q N} \right\} 
= \Phi(t) + o(1),
\]
for appropriate constants $\mu$ and $\sigma^2$.

Our first observation is that $f(n)$ given in \eqref{eqdeffn} is
$q$-quasi-additive. 
The logarithmic growth property is trivially satisfied since
$s_q(n)\le (q-1)\log_q n$. 

\begin{lemma}
Let $A$ and $r$ be positive integers with $q^r \ge A$. Then $g(n) = s_q(An)$
is $q$-quasi-additive {\em(}with parameter $r${\em)}.
\end{lemma}

\begin{proof}
Suppose that $b< q^k$. Then $A b < q^{k+r}$, and consequently
\[
g(q^{k+r}a+b) = s_q(q^{k+r}Aa+Ab) = s_q(Aa) +s_q(Ab) = g(a) + g(b).
\qedhere\]
\end{proof}

Since linear combinations of $q$-quasi-additive functions are
$q$-quasi-additive, 
it directly follows that 
$f(n)$, 
as given by \eqref{eqdeffn}, satisfies
a central limit theorem of the prescribed form. 
This also implies that the vector 
\eqref{eqTh} satisfies a $d$-dimensional central limit theorem with asymptotic mean 
$((q-1)/2,\ldots, (q-1)/2) \cdot \log_q N$.
(This follows 
from the fact that a random vector ${\bf X} = (X_1,\ldots X_d)$ is Gau\ss ian with mean vector $(\mu_1,\ldots, \mu_d)$
and covariance matrix $\Sigma$ if and only if every projection $c_1 X_1 + \cdots + c_d X_d$ with
real $c_1,\ldots, c_d$ is univariate Gau\ss ian with mean $c_1 \mu_1 + \cdots + c_d \mu_d$ and variance
$(c_1, \dots, c_d) \Sigma (c_1, \dots, c_d)^t$). 

It remains to compute the covariance matrix in the case, where
$q$ is a prime number. 

\begin{lemma}\label{Leasymp1}
Let $q\ge 2$ be a prime number, let $A_1,A_2$ be positive integers
that are not 
divisible by $q$, and set $D = {\rm gcd}(A_1,A_2)$. 
Then, uniformly for $(\log N)^{1/3} \le i,j \le \log_q N - (\log N)^{1/3}$
and $a,b\in \{0,1,\ldots,q-1\}$, we have
\begin{multline*}
\frac 1N \# \left\{ n < N : \varepsilon_i(A_1n) = a,\,
\varepsilon_j(A_2n) = b \right\} \\
= \begin{cases}
\frac 1{q^2} + O\left( (\log N)^{-C} \right), & \mbox{if\/ $i\ne j$,} \\
\frac 1{q^2} + \frac{D^2}{A_1 A_2} 
\sum\limits_{\ell \ne 0}  \frac {1} {4\pi^2 \ell^2}
{\left( e\left( -\frac{\ell A_2a}{qD} \right)- e\left( -\frac{\ell
    A_2(a+1)}{qD}  \right)\right)  
\left( e\left( \frac{\ell A_1b}{qD} \right)- e\left( -\frac{\ell
  A_1(b+1)}{qD}  \right)\right)} 
\kern-1.5cm
 & \\
\kern2cm +
O\left( (\log N)^{-C} \right) ,& \mbox{if\/ $i = j$,} 
\end{cases}
\end{multline*}
for any given $C> 0$. Here, $e(x) = e^{2\pi i x}$.
\end{lemma}

\begin{proof}
We adapt the method of \cite{BK} to the present situation. However, in
order to make the presentation more 
transparent, we first present a slightly simplified approach. First we
note that $\varepsilon_j(n) = a$ if and only if $\{n q^{-j-1} \} \in 
\left[  a/q, (a+1)/q \right)$, 
where $\{x\} = x - \lfloor x \rfloor$ denotes the fractional part of
$x$. We also note that the 
Fourier series of the characteristic function ${\bf
  1}_{[a/q,(a+1)/q)}(x)$ is given by 
\[
{\bf 1}_{\left[ \frac aq,\frac{a+1}q\right)}(x) = \sum_{m} d_m(a)
  e(mx) \quad\mbox{with}\quad 
d_m(a) = \begin{cases}
\frac 1q, & \mbox{if $m=0$}, \\
\frac{ e\left( -\frac {ma}q \right) - e\left( -\frac {m(a+1)}q
  \right)}{2\pi i m},& \mbox{if $m\ne 0$}. 
\end{cases}
\]
This Fourier series is not absolutely convergent. This is the major
reason that we have 
to be more precise in a second round. Observe that $d_m(a) = 0$ if
$m\ne 0$ and $m \equiv 0 \bmod q$. 

We have
\begin{multline*}
\# \left\{ n < N : \varepsilon_i(A_1n) = a,\, \varepsilon_j(A_2n) = b \right\}
= \sum_{n< N} {\bf 1}_{\left[ \frac aq,\frac{a+1}q\right)}\left(
  \frac{A_1 n}{q^{i+1}} \right)  
{\bf 1}_{\left[ \frac bq,\frac{b+1}q\right)}\left( \frac{A_2
    n}{q^{j+1}} \right) \\ 
 = \sum_{m_1,m_2} d_{m_1}(a)\,d_{m_2}(b) \sum_{n< N} e\left( \left(
 \frac{A_1 m_1}{q^{i+1}} +\frac{A_2 m_2}{q^{j+1}}  \right) n\right). 
\end{multline*}
Since 
\[
\left|\sum_{n< N} e(\alpha n)\right|\le \frac 2{|1-e(\alpha)|},
\]
we may {\it neglect\/} all exponential sums where
$\alpha = \frac{A_1 m_1}{q^{i+1}} +\frac{A_2 m_2}{q^{j+1}}$
is not an integer.
At this moment, this step is not rigorous since 
the Fourier series is not absolutely convergent. 

Next suppose that $\alpha$ is an integer. 
If $i\ne j$, the number $\frac{A_1 m_1}{q^{i+1}} +\frac{A_2
  m_2}{q^{j+1}}$ can be an integer only if
$m_1 = m_2 = 0$ since we also assume  
that $A_1$ and $A_2$ are not divisible by $q$. 
Thus we should get
\[
\# \left\{ n < N : \varepsilon_i(A_1n) = a,\, \varepsilon_j(A_2n) = b
\right\} =  d_0(a)d_0(b) N + o(N) = \frac N{q^2} + o(N). 
\]
If $i=j$, then the assumption 
$\frac{A_1 m_1}{q^{j+1}} +\frac{A_2 m_2}{q^{j+1}} = k$
for an integer~$k$ leads to
%
$|m_1| \ge \frac {1} {2A_1} |k| q^{i+1} \ge
\frac {1} {2A_1}  |k| q^{( \log N)^{1/3}}$ or $|m_2|
\ge \frac {1} {2A_2} |k| q^{j+1} \ge
\frac {1} {2A_2}  |k| q^{( \log N)^{1/3}}$ 
so that the corresponding terms are negligible (if the Fourier series
would be absolutely convergent). 

Thus we should get (again by observing that all
summands for which $\alpha$ is an integer 
can be put into an error term)
\begin{multline*}
\# \left\{ n < N : \varepsilon_j(A_1n) = a,\, \varepsilon_j(A_2n) = b
  \right\}  
= \sum_{\ell} d_{\ell A_2/D}(a) d_{-\ell A_1/D}(b) \, N + o(N) \\
\quad= \frac {N}{q^2} + \sum_{\ell \ne 0} 
\frac 1{4\pi^2 \ell^2}
{\left( e\left( -\tfrac{\ell A_2a}{qD} \right)- e\left( -\tfrac{\ell
    A_2(a+1)}{qD}  \right)\right)  
\left( e\left( \tfrac{\ell A_1b}{qD} \right)- e\left( -\tfrac{\ell
  A_1(b+1)}{qD}  \right)\right)} 
  \, N\\
 + o(N).
\end{multline*}

In order to make the above heuristics rigorous, 
we proceed as in \cite{BK}. We replace the characteristic
function ${\bf 1}_{[a/q,(a+1)/q)}(x)$ by a smoothed version.
Let $\chi_{a,\Delta}(x)$ be defined by
\[
\chi_{a,\Delta}(x) := \frac 1{\Delta} \int_{-\Delta/2}^{\Delta/2} 
{\bf 1}_{[\frac aq,\frac{a+1}q]}(\{x+z\})\, dz,
\]
The Fourier coefficients of the Fourier 
series $\chi_{a,\Delta}(x) = \sum_{m\in\mathbb{Z}} d_{m,\Delta}(a) e(mx)$
are given by
\[
d_{0,\Delta}(a) = \frac 1q,
\]
and for $m\ne 0$ by 
\[
d_{m,\Delta}(a) =  \frac{e\left(-\frac {ma}q\right)-
e\left(-\frac {m(a+1)}q\right)}{2\pi i m} \cdot
\frac{e\left(\frac {m\Delta}2\right)-e\left(-\frac {m\Delta}2\right)}
{2\pi i m\Delta}.
\]
Note that $d_{m,\Delta}(a) = 0$ if $m\ne 0$ and $m\equiv 0\bmod q$, and that
\[
|d_{m,\Delta}(a)| \le \min\left( \frac 1{\pi |m|},\frac 1{\Delta\pi
  m^2}\right).  
\]
By definition, we have $0\le \chi_{a,\Delta}(x)\le 1$ and
\[
\chi_{a,\Delta}(x) = \begin{cases}
1, & \mbox{if } x\in \left[ \frac aq + \Delta, \frac{a+1}q
  -\Delta\right], \\ 
0, & \mbox{if } x\in [0,1] \setminus \left[ \frac aq - \Delta,
  \frac{a+1}q +\Delta\right]. 
\end{cases}
\]
In particular, we set $\Delta = (\log N)^{-C}$ for some (sufficiently
large) constant $C$. 
Of course we have to take into account all error
terms. The smoothing error can be handled with the help of a
discrepancy estimate (see \cite{BK}). 
Putting the resulting estimates together --- we leave the details to
the reader ---, one obtains  
\begin{align*}
\# \left\{ n < N : \varepsilon_i(A_1n) = a,\, \varepsilon_j(A_2n) = b
\right\} &=  d_{0,\Delta}(a)d_{0,\Delta}(b) N + 
O\left( N (\log N)^{-C} \right) \\
&= \frac{N}{q^2} + O\left( N (\log N)^{-C} \right) 
\end{align*}
for $i\ne j$, and 
\begin{multline*}
\# \left\{ n < N : \varepsilon_j(A_1n) = a,\, \varepsilon_j(A_2n) = b \right\} 
= \sum_{\ell} d_{\ell A_2/D,\Delta}(a) d_{-\ell A_1/D,\Delta}(b) \, N \\
+  O\left( N (\log N)^{-C} \right)
\end{multline*}
for $i=j$,
where all estimates are uniform for $(\log N)^{1/3} \le i,j \le \log_q
N - (\log N)^{1/3}$. 
Since
\[
d_{m,\Delta}(a) = d_m(a)\frac{\sin(\pi m \Delta)}{\pi m \Delta} = 
d_m(a)\left( 1 + O\left( \frac 1{m\Delta} \right) \right) 
\]
for $1\le |m| \le 1/\Delta$, we obtain (with $A = \max\{A_1,A_2\}$)
\begin{multline*}
\sum_{1\le |\ell| \le 1/(A \Delta)} d_{\ell A_2/D,\Delta}(a) d_{-\ell
  A_1/D,\Delta}(b)  
= \sum_{1\le |\ell| \le 1/(A \Delta)} d_{\ell A_2/D}(a) d_{-\ell A_1/D}(b) \\
+ O\left(  \Delta \log(1/\Delta)   \right),
\end{multline*}
and 
\begin{align*}
\sum_{ |\ell| > 1/(A \Delta)} d_{\ell A_2/D,\Delta}(a) d_{-\ell
  A_1/D,\Delta}(b)  
&= O\left( \Delta \right),\\
\sum_{ |\ell| > 1/(A \Delta)} d_{\ell A_2/D}(a) d_{-\ell A_1/D}(b) 
&= O\left( \Delta \right).
\end{align*}
Thus, 
\begin{equation*}
\sum_{\ell} d_{\ell A_2/D,\Delta}(a) d_{-\ell A_1/D,\Delta}(b) 
= \sum_{\ell} d_{\ell A_2/D}(a) d_{-\ell A_1/D}(b) 
+ O\left( \frac{\log\log N}{(\log N)^C}  \right)
\end{equation*}
This completes the proof of the lemma.
\end{proof}

It is now not difficult to complete the computation of the covariance matrix
(which also completes the Proof of Theorem~\ref{ThCLT}).
By definition, the covariance of $s_q(A_1n)$ and $s_q(A_2n)$ is given by
\[
\Cov = \frac 1N \sum_{n<N} s_q(A_1n) s_q(A_2n) -  \frac 1N \sum_{n<N}
s_q(A_1n) \cdot \frac 1N \sum_{n<N} s_q(A_2n). 
\]
In order to apply Lemma~\ref{Leasymp1}, we neglect the digits
$\varepsilon_j$ with $j\le (\log N)^{1/3}$ 
or $j \ge  \log_q(N) -(\log N)^{1/3}$ and 
denote by $\overline s_q$ the sum of digits of the remaining digits
$\varepsilon_j$ with 
$(\log N)^{1/3} <  j < \log_q N - (\log N)^{1/3}$. Then the
corresponding approximate covariance $\overline{\Cov}$ satisfies 
\[
\Cov - \overline{\Cov} = O\left( (\log N)^{5/6} \right),
\]
which can be shown with the help of the Cauchy--Schwarz inequality.
Hence, by rewriting $\overline{\Cov}$ with the help of the numbers 
$\frac 1N \# \left\{ n < N : \varepsilon_i(A_1n) = a,\,
\varepsilon_j(A_2n) = b \right\}$ (from Lemma~\ref{Leasymp1}) and  
the numbers
\[
\frac 1N \# \left\{ n < N : \varepsilon_j(A_i n) = a \right\} = \frac
1q + O\left( (\log N)^{-C} \right) 
\]
(note that the fact that this asymptotic property 
holds uniformly for 
$(\log N)^{1/3}\le j\le \log_q N - (\log N)^{1/3}$, $a\in
\{0,1,\ldots, q-1\}$,  
and $i = 1,2$ follows from  Lemma~\ref{Leasymp1}), we immediately get 
\begin{align*}
\overline {\Cov} &= L \frac{D^2}{A_1 A_2} \sum_{a,b=0}^{q-1}  ab
\sum_{\ell \not \equiv 0 \bmod q} 
\frac 1{4\pi^2 \ell^2}\\
&\kern4cm
\cdot
{\left( e\left( -\tfrac{\ell A_2a}{qD} \right)- e\left( -\tfrac{\ell
    A_2(a+1)}{qD}  \right)\right)  
\left( e\left( \tfrac{\ell A_1b}{qD} \right)- e\left( -\tfrac{\ell
  A_1(b+1)}{qD}  \right)\right)} 
 \\
&\kern4cm
+ O\left( (\log N)^{2-C} \right) \\
&= L \frac{D^2}{A_1 A_2} \frac{q^2}{4 \pi^2}  \sum_{\ell \not \equiv 0
  \bmod q} \frac 1{\ell^2}  
+ O\left( (\log N)^{2-C} \right) \\
&= L \frac{D^2}{A_1 A_2} \frac{q^2-1}{12} + O\left( (\log N)^{2-C} \right),
\end{align*}
where $L = \lfloor \log_q N - 2(\log N)^{1/3} \rfloor$, 
and where we have used the identity
\[
\sum_{a=0}^{q-1} a\,  e(a k/q) = \frac q{e(k/q) -1},
\] 
which is valid for integers $k$ that are not divisible by $q$.
We can choose $C$ appropriately --- for example $C = 2$ --- and
finally obtain 
\[
\Cov = \frac {q^2-1}{12} \frac{{\rm gcd}(A_i,A_j)^2}{A_i A_j}\log_q N
+ O\left( (\log N)^{5/6} \right), 
\]
which completes the proof of Theorem~\ref{ThCLT}.

\section{Asymptotic divisibility of Catalan-like integer sequences}
\label{sec2}

The main result in this section concerns divisibility of
``Catalan-like" sequences by prime powers.

\begin{theo} \label{thm:almostall}
Let $p$ be a given prime number, $\alpha$ a positive integer, 
$P(n)$ a polynomial in~$n$
with integer coefficients, and
$(C_i)_{1\le i\le r}$,
$(D_i)_{1\le i\le s}$,
$(E_i)_{1\le i\le t}$, 
$(F_i)_{1\le i\le t}$ given integer sequences with $C_i,D_i>0$
and $p\nmid\gcd(E_i,F_i)$ for all~$i$, $\sum_{i=1}^rC_i=
\sum_{i=1}^sD_i$, and\break
$\{C_i:1\le i\le r\}\ne\{D_i:1\le i\le s\}$.
If all elements of the sequence $\big(S(n)\big)_{n\ge0}$, defined by
\begin{equation} \label{eq:Cat}
S(n):=\frac {P(n)} {
\prod _{i=1} ^{t}(E_in+F_i)}
\frac {
\prod _{i=1} ^{r}(C_in)!} {
\prod _{i=1} ^{s}(D_in)!} ,
\end{equation}
are integers, then 
\begin{equation} \label{eq:asydiv} 
\lim_{N\to\infty}\frac {1} {N}\# \left\{ n < N : 
S(n)\equiv0~(\text{\em mod }p^\alpha)\right\}=1.
\end{equation}
\end{theo}

We note that (\ref{eq:asydiv}) remains true if $\alpha$
increases slowly with $N$. In particular, we can choose
$\alpha = \lfloor \eta \log N \rfloor$ for an appropriate $\eta > 0$.
We will actually show 
in the proof of Theorem~\ref{thm:almostall} that
\[
\# \left\{ n < N : v_p( S(n)) < \eta \log N \right\} = O\left( \frac{N}{\sqrt{\log N}} \right)
\]
if $\eta > 0$ is sufficiently small.

Furthermore we note that the assumption $p\nmid\gcd(E_i,F_i)$ is not
really necessary since we can always reduce the problem to this case by
separating the factors $p^{v_p(\gcd(E_i,F_i))}$.
Thus, we immediately obtain the following corollary.

\begin{cor} \label{cor:almostall}
Let $S(n)$ be given as in Theorem~\ref{thm:almostall}
(without assuming the condition $p\nmid\gcd(E_i,F_i)$).
Then, for all positive integers $m$, we have
\begin{equation} \label{eq:asydiv-2} 
\lim_{N\to\infty}\frac {1} {N}\# \left\{ n < N : 
S(n)\equiv0~(\text{\em mod }m)\right\}=1.
\end{equation}
\end{cor}

We call integer sequences of the form as in \eqref{eq:Cat}
--- that is, integer sequences given by a product/quotient of
factorials multiplied by a rational function ---
``Catalan-like" since the Catalan numbers $\frac {1} {n+1}\binom
{2n}n$ represent obviously such a sequence, but as well many
other sequences that one finds in the literature (and in the
On-Line Encyclopedia of Integer Sequences \cite{OEIS}).

\begin{exs}
All of the following sequences are ``Catalan-like" in the sense
of \eqref{eq:Cat}.

\medskip
(1) {\it Binomial coefficients} such as the {\it central binomial
    coefficients} $\binom {2n}n$, or more generally $\binom
  {(a+b)n}{an}$ for positive integers $a$ and~$b$, including
variations such as $\binom {2n}{n-1}$, etc.

\medskip
(2) {\em Multinomial coefficients} such as $\frac {((a_1+a_2+\dots+a_s)n)!} 
{(a_1n)!\,(a_2n)!\cdots(a_sn)!}$, etc.

\medskip
(3) {\it Fu\ss--Catalan numbers}. These are defined by
(cf.\ \cite[pp.~59--60]{ArmDAA})
$\frac {1} {n}\binom {(m+1)n}{n-1}$, where $m$ is a given positive integer.

(4) Gessel's \cite{GessAW} {\it super ballot numbers}
(often also called {\it super-Catalan numbers})\break $\frac {(2n)!\,(2m)!} 
{n!\,m!\,(m+n)!}$ for non-negative integers~$m$, or for $m=an$
with $a$ a positive integer.

(5) Many counting sequences in {\it tree} and {\it map enumeration}
(cf.\ \cite{SchaAB} for a survey) such as
$\frac{m+1}{n((m-1)n +2)} \binom{mn}{n-1}$ 
($m$-ary blossom trees with $n$ white nodes;
cf.\ \cite[Sec.~3]{SchaAA}), $\frac {2\cdot3^n} {(n+2)(n+1)}\binom {2n}n$
(number of rooted planar maps with $n$ edges; cf.\ \cite{TuttAC}),
$\frac {2} {(3n-1)(3n-2)}\binom {3n-1}n$ (number of rooted non-separable 
planar maps with $n$ edges; cf.\ \cite{BrowAA}),
$\frac {2} {(3n+1)(n+1)}\binom {4n+1}n$ (number of rooted planar
triangulations with $n+3$ vertices; cf.\ \cite{TuttAA}),
$\frac {1} {2(n+2)(n+1)}\binom {2n}n\binom {2n+2}{n+1}$ (number of 
rooted Hamiltonian maps with $2n$ vertices; cf.\ \cite{TuttAB}),
to mention just a few.
\end{exs}

What Theorem~\ref{thm:almostall} says is that, for any of these
sequences, most elements (in the sense that the proportion of
those in the set of all elements tends to~$1$) are divisible by~$p^\alpha$, for
a given prime number~$p$ and given positive integer~$\alpha$.

We should at this point remind the reader of {\it Landau's criterion}
\cite{LandAA}, which says that
$$
U(n):=\frac {
\prod _{i=1} ^{r}(C_in)!} {
\prod _{i=1} ^{s}(D_in)!} 
$$
is an integer for all $n$ if and only if 
\begin{equation} \label{eq:Landau} 
\sum_{i=1}^r\fl{C_ix}
-\sum_{i=1}^s\fl{D_ix}\ge0
\end{equation}
for all real numbers~$x$. (Here, we still assume that $\sum_{i=1}^rC_i=
\sum_{i=1}^sD_i$.)
In view of \cite[Lemma~3.3]{BobeAA}, which says that if $U(n)$
is non-integral for some~$n$ then, for almost all primes~$p$,
there exists an~$n$ such that $v_p\big(U(n)\big)<0$, this means
(more or less; we do not believe that the polynomial $P(n)$
can ``correct" non-integrality of $U(n)$ for all~$n$) that \eqref{eq:Landau}
is an implicit assumption in Theorem~\ref{thm:almostall}.

\medskip
For the proof of Theorem~\ref{thm:almostall}, we consider
the integer interval $[0,N-1]$ as a probability space, with each
integer equally likely, precisely as in Section~\ref{sec1}.
For notational convenience, the corresponding probability function will
be denoted by $\PP_N$.
Functions on the non-negative integers 
are then interpreted as random variables $X$ on this space by
restricting them to $[0,N-1]$. The expectation of $X$ on
the space, that is, $\frac {1} {N}\sum_{i=0}^{N-1}X(i)$, will be
denoted by $\E_N(X)$, the variance will be denoted by $\V_N(X)$,
and the covariance of two functions by $\Cov_N(X,Y)$.

We need two auxiliary lemmas.
The first concerns asymptotic mean and variance for the $p$-adic
valuation of a linear function.

\begin{lemma} \label{lem:1}
Let $E$ and $F$ be integers, $E>0$, not both divisible by the prime~$p$.
If $v_p(En+B)$ is considered as a random variable for $n$
in the integer interval $[0,N-1]$, then
\begin{equation} \label{eq:linearE} 
\E_N\!\big(v_p(En+F)\big)=\begin{cases} 
0,&\text{if\/ }p\mid E,\\
\frac {1} {p-1}+o(1),&\text{if\/ }p\nmid E,
\end{cases}
\quad \quad \text{as }N\to\infty,
\end{equation}
and
\begin{equation} \label{eq:linearV} 
\V_N\!\big(v_p(En+F)\big)=\begin{cases} 
0,&\text{if\/ }p\mid E,\\
\frac {p} {(p-1)^2}+o(1),&\text{if\/ }p\nmid E,
\end{cases}
\quad \quad \text{as }N\to\infty.
\end{equation}
\end{lemma}

\begin{proof}
The first assertion in the case distinction in \eqref{eq:linearE} is
obvious since our assumptions imply that
$En+F\hbox{${}\not\equiv{}$}0$~(mod~$p$) if $p\mid E$. 
If $p\nmid E$, then the congruence $En+F\equiv 0$~(mod~$p^\alpha$)
has a unique solution for $n$ modulo~$p^\alpha$ for any given
positive integer~$\alpha$. Thus, we have
$$
\E_N\!\big(v_p(an+b)\big)=\frac {1} {N}
\sum_{\ell= 1}^{ \lfloor  \log_p N \rfloor}\left(\frac {N} {p^\ell} +
O(1)\right) 
=\frac {1} {p-1}+o(1),\quad \quad \text{as }N\to\infty.
$$
Similarly, still assuming $p\nmid E$, for the variance we have
\begin{align*}
\V_N\!\big(v_p(an+b)\big)&=\frac {1} {N}
\sum_{\ell = 1}^{ \lfloor  \log_p N \rfloor}(2\ell-1)\left(\frac {N}
    {p^\ell} + O(1)\right) 
-\Big(\E_N\!\big(v_p(an+b)\big)\Big)^2\\
&=\frac {p+1} {(p-1)^2}-\frac {1} {(p-1)^2}+o(1),
\quad \quad \text{as }N\to\infty,
\end{align*}
establishing also \eqref{eq:linearV}.
\end{proof}

The second auxiliary lemma provides an asymptotic upper bound on the
covariance of a linear function and the sum-of-digits function
of a linear function.

\begin{lemma} \label{lem:2}
Let $C$, $E$, and $F$ be integers, $C,E>0$, 
and $E$ and $F$ not both divisible by $p$. 
If $s_p(Cn)$ and $v_p(En+B)$ are considered as random variables for $n$
in the integer interval $[0,N-1]$, then
\begin{equation} \label{eq:cov} 
\Cov_N(s_p(Cn),v_p(En+F))= O\!\left(\log^{1/2}_p(N)\right), \quad \quad 
\text{as }N\to\infty.
\end{equation}
\end{lemma}

\begin{proof}
By the Cauchy--Schwarz inequality in probabilistic setting, we have
$$
\Cov_N\!\Big(s_p(Cn),v_p\big(En+F\big)\Big)\le 
\V_N\!\big(s_p(Cn)\big)^{1/2}\V_N\!\Big(v_p\big(En+F\big)\Big)^{1/2}.
$$
The variance of $s_p(Cn)=s_p(Cp^{-v_p(C)}n)$ has been (implicitly) given
in \eqref{eq:clt} (see the line below that equation; see also
\eqref{eq:Cov} with $q=p$ and $A_i=A_j=Cp^{-v_p(C)}$) and turned out
to be of the order $\log_p(N)$, while the
variance of $v_p\big(En+F\big)$ has been given in \eqref{eq:linearV}
and turned out to be bounded.
The assertion of the lemma is hence obvious.
\end{proof}

\begin{proof}[Proof of Theorem~\ref{thm:almostall}]
With $S(n)$ given in \eqref{eq:Cat}, we have
\begin{align} \notag
v_p\big(S(n)\big)&=
v_p\big(P(n)\big)-
\sum_{i=1} ^{t}v_p(E_in+F_i)+
\sum _{i=1} ^{r}v_p\big((C_in)!\big)-
\sum _{i=1} ^{s}v_p\big((D_in)!\big)\\
&\ge
-\sum_{i=1} ^{t}v_p(E_in+F_i)-
\frac {1} {p-1}\sum _{i=1} ^{r}s_p(C_in)+
\frac {1} {p-1}\sum _{i=1} ^{s}s_p(D_in).
\label{eq:S(n)}
\end{align}
Here, we used Legendre's formula \eqref{eq:Leg} and the
assumption that $\sum_{i=1}^rC_i=
\sum_{i=1}^sD_i$.

Now, it follows from \cite[Lemma~3.5 and its proof]{BobeAA},
that under the integrality and non-triviality assumption for $S(n)$
of the theorem, we have $r<s$. 

The expression on the right-hand side of \eqref{eq:S(n)} is
a linear combination of the functions $v_p(E_in+F_i)$,
$s_p(C_in)$, and $s_p(D_in)$, which we view again as random variables
on\break $[0,N-1]$. For convenience, let us denote the function on
the right-hand side of \eqref{eq:S(n)} by $T(n)$. By Theorem~\ref{ThCLT}
and \eqref{eq:linearE}, we have
$$
\E_N\!\big(T(n)\big)=\Omega\big(\log_p(N)\big),\quad \quad \text{as }N\to\infty.
$$
The reader should observe that the inequality $r<s$ is used here crucially.
On the other hand, the variance of $T(n)$ is bounded above by the sum
of the pairwise covariances of the involved random variables
(functions).
By Theorem~\ref{ThCLT}, 
\eqref{eq:linearV}, and
\eqref{eq:cov}, we see that
$$
\V_N\!\big(T(n)\big)=O\big(\log_p(N)\big),\quad \quad \text{as }N\to\infty.
$$

Given a random variable $X$, Chebyshev's inequality reads
\begin{equation} \label{eq:Cheb}
\PP\!\big(\vert X-\E(X)\vert<\ep\big)> 1-\frac {1} {\ep^2}V(X). 
\end{equation}
Choosing $\ep=\big(\log_p(n)\big)^{3/4}$ and $X=T(n)$, we get
$$
\PP_N\!\left(\left\vert T(n)-\E_N\!\big(T(n)\big)\right\vert<\log^{3/4}_p(N)\right)
=1+O\!\(\log^{-1/2}_p(N)\),\quad \quad \text{as }N\to\infty.
$$
Thus, for $N$ large enough, we have
$$
T(n)>\E_N\!\big(T(n)\big)-\log^{3/4}_p(N)=\Omega\!\(\log_p(N)\)>\alpha,
$$
with probability $1+O\big(\log^{-1/2}_p(N)\big)$.
If we use this information in \eqref{eq:S(n)}, then the assertion of 
the theorem follows immediately.
\end{proof}

\end{document}